\documentclass[preprint, 12pt, english]{elsarticle}
\usepackage{amsmath}
\usepackage{latexsym, amssymb, mathtools}
\usepackage{amsthm}
\usepackage{color}
\usepackage{txfonts}

\newtheorem{thm}{Theorem}[section] 

\newtheorem{exmpl}[thm]{Example}
\newtheorem{lem}[thm]{Lemma}
\newtheorem{prop}[thm]{Proposition}

\newtheorem{rem}[thm]{Remark}

\newcommand\operA[2]{{\if!#2!\operatorname{#1}\else{\operatorname{#1}_{#2}^{\phantom{I}}}\fi}} 

\newcommand\Cref[1]{{Corollary~\ref{#1}}}%
\newcommand{\Trace}[1][]{\if!#1!\operatorname{Tr}\else{\operatorname{Tr}_{#1}^{\phantom{I}}}\fi} 

\long\def\forget#1\forgotten{{}} %

\def\({\left(}
\def\){\right)}

\usepackage{tabu}

\newif\iffurther
\furtherfalse

\newif\ifXY 
\XYtrue     
\ifXY

\input xy
\input xyidioms.tex
\usepackage{xy}
\xyoption{all} %
\fi 

\usepackage{babel}

\journal{??}

\begin{document}

\begin{frontmatter}

\title{Common subfields of p-algebras of prime degree}

\author{Adam Chapman}
\ead{adam1chapman@yahoo.com}

\address{Department of Mathematics, Michigan State University, East Lansing, MI 48824}

\begin{abstract}
We show that if two division $p$-algebras of prime degree share an inseparable field extension of the center then they also share a cyclic separable one. We show that the converse is in general not true. We also point out that sharing all the inseparable field extensions of the center does not imply sharing all the cyclic separable ones.
\end{abstract}


\end{frontmatter}

\section{Introduction}

A $p$-algebra is a central simple algebra of prime power degree $p^m$ over a field $F$ of characteristic $p$.
If a $p$-algebra is cyclic of prime degree $p$ then it has a symbol presentation $[\alpha,\beta)_{p,F}$ for some $\alpha \in F$ and $\beta \in F^\times$ which stands for the algebra generated over $F$ by $x$ and $y$ subject to the relations $x^p-x=\alpha$, $y^p=\beta$ and $y x y^{-1}=x+1$.

The symbol presentation of an algebra is not unique. For example $[\alpha,\beta)_{p,F}$ and $[\alpha+\beta,\beta)_{p,F}$ present the same algebra.
We say that two algebras are right linked if they have symbol presentations sharing the right slot, e.g. $[\alpha,\beta)_{p,F}$ and $[\gamma,\beta)_{p,F}$ are right linked.
We say that two algebras are left linked if they have symbol presentations sharing the left slot, e.g. $[\alpha,\beta)_{p,F}$ and $[\alpha,\delta)_{p,F}$ are left linked.

In \cite{Draxl:quatsep} Draxl proved that if two $p$-algebras of degree $2$ (i.e. quaternion algebras) are right linked then they are also left linked.
In \citep{Lam:char2quat} Lam gave a simpler proof and showed that the converse is not true. In \citep{ElduqueVilla:2005} this was generalized to Hurwitz algebras, and in \cite{fairve:thesis} this was translated and generalized to arbitrary $n$-fold Pfister forms.

In this note we prove that being right linked implies being left linked for any prime $p$ and that the converse is in general not true.
We also point out that sharing all the inseparable field extensions of the center does not imply sharing all the cyclic separable ones.

\section{Artin-Schreier and $p$-central elements}

Given a $p$-algebra $A$ over $F$, an element $x \in A$ satisfying $x^p-x \in F$ is called Artin-Schreier, and an element $y \in A$ satisfying $y^p \in F$ is called $p$-central.
The Artin-Schreier elements are exactly the elements generating cyclic separable field extensions over the center, and the $p$-central elements are exactly the elements generating pure inseparable fields extensions over the center

If two nonzero elements $x$ and $y$ in $A$ satisfy $y x y^{-1}=x+1$ then $x$ is Artin-Schreier, $y$ is $p$-central and $A=[x^p-x,y^p)_{p,F}$. For any Artin-Schreier element $x \in A$ there exists an element $y \in A$ such that $y x y^{-1}=x+1$ (see \cite[Chapter 5, Theorem 9]{Albert:1968}). Similarly, for any $p$-central element $y \in A$ there exists an element $x \in A$ such that $y x y^{-1}=x+1$ (see \cite[Chapter 7, Lemma 10]{Albert:1968}). 

Therefore, two $p$-algebras of prime degree $p$ over $F$ contain a common purely inseparable degree $p$ field extension of the center if and only if they are right linked.
They contain a common cyclic separable degree $p$ extension of the center if and only if they are left linked.

\section{Right linkage implies left linkage}

We shall now show that being right linked implies being left linked.

\begin{rem}
In \cite[Lemma 2.2]{ChapmanKuo2015} it was observed that if $x$ is Artin-Schreier then every other element $t$ decomposes as $t_0+t_1+\dots+t_{p-1}$ where $t_i x-x t_i=i t_i$ for each $0 \leq i \leq p-1$.
This decomposition is unique:
Assume $t$ decomposes both as $t_0+t_1+\dots+t_{p-1}$ and $t_0'+t_1'+\dots+t_{p-1}'$. Then 
$$t_0+t_1+\dots+t_{p-1}=t_0'+t_1'+\dots+t_{p-1}'.$$
By computing $t x-x t$ for both decompositions we obtain
$$t_1+2 t_1+\dots+(p-1) t_{p-1}=t_1'+2 t_2'+\dots+(p-1) t_{p-1}'.$$
By repeating this process several times, we end up with a system of $p$ linearly independent equations, and the solution to this system is $t_i=t_i'$ for each $0 \leq i \leq p$.
The decomposition $t=t_0+\dots+t_{p-1}$ is in fact the eigenvector decomposition of $t$ with respect to the linear transformation 
$$v \rightarrow v x-x v$$
with each $t_i$ lying in the eigenspace of the eigenvalue $i$.
\end{rem}

\begin{lem}\label{Formula}
In a division $p$-algebra $A$ of prime degree $p$ over $F$, if two nonzero elements $x$ and $y$ satisfy $y x-x y=k y$ for some $1 \leq k \leq p-1$, then $(x+y)^p-(x+y)=x^p-x+y^p$.
\end{lem}

\begin{proof}
The element $x+y$ is Artin-Schreier, because it satisfies $y^m (x+y) y^{-m}=x+y+1$ where $m$ is the unique integer satisfying $mk \equiv 1 \pmod{p}$. Therefore $(x+y)^p-(x+y) \in F$. According to the previous remark, there is a unique decomposition of $(x+y)^p-(x+y)$ as $t_0+\dots+t_{p-1}$ where $t_i x-x t_i=i t_i$ for each $0 \leq i \leq p-1$. In this case, $t_i=x^{p-i} * y^i$ for each $k \neq i \geq 1$, $t_k=x^{p-1} * y-y$ and $t_0=x^p+y^p-x$, where $x^r * y^s$ stands for the sum of all the words in which $x$ appears $r$ times and $y$ appears $s$ times. Since $(x+y)^p-(x+y) \in F$, we have $(x+y)^p-(x+y)=t_0$.
\end{proof}

\begin{thm}
If two division $p$-algebras of prime degree $p$ are right linked then they are left linked.
\end{thm}

\begin{proof}
Let $A=F[x,y : x^p-x=\alpha, y^p=\beta, y x y^{-1}=x+1]$ and $A'=F[x',y' : x'^p-x'=\gamma, y'^p=\beta, y' x' y'^{-1}=x'+1]$ be two right linked algebras.
Let $\lambda$ be the unique solution to the linear equation $\alpha+\beta (\alpha-\lambda)=\gamma$ over $F$.
Since $z=x+\lambda y+x y$  satisfies $(\lambda y+x y) z-z (\lambda y+x y)=\lambda y+x y$, according to Lemma \ref{Formula}, $z$ is Artin-Schreier in $A$ satisfying $z^p-z=(x^p-x)+(\lambda y+x y)^p=\alpha+N_{F[x]/F}(\lambda+x) y^p=\alpha+(\alpha+\lambda^p-\lambda) \beta$. Therefore $A$ has the symbol presentation $A=[\alpha+(\alpha+\lambda^p-\lambda) \beta,(\alpha+\lambda^p-\lambda) \beta)_{p,F}=[\gamma +\lambda^p \beta,(\alpha+\lambda^p-\lambda) \beta)_{p,F}$.

Another way of seeing that $A=[\alpha+(\alpha+\lambda^p-\lambda) \beta,(\alpha+\lambda^p-\lambda) \beta)_{p,F}$ is by noticing that $[\alpha+(\alpha+\lambda^p-\lambda) \beta,(\alpha+\lambda^p-\lambda) \beta)_{p,F}$ is Brauer equivalent to $$[\alpha,(\alpha+\lambda^p-\lambda) \beta)_{p,F} \otimes [(\alpha+\lambda^p-\lambda) \beta,(\alpha+\lambda^p-\lambda) \beta)_{p,F}.$$ The second algebra in this product is split and the first one is isomorphic to $[\alpha,\beta)_{p,F}$.

Similarly, $z'=x'+\lambda y'$ is Artin-Schreier in $B$ satisfying $y' z'-z' y'=y'$ and $z'^p-z'=\gamma+\lambda^p \beta$, and therefore $A'=[\gamma+\lambda^p \beta,\beta)_{p,F}$.
Consequently $A$ and $A'$ are also left linked.
\end{proof}

\section{Counterexample for the converse}
We want to construct a counterexample for the converse.
Over fields of imperfect exponent 1, we cannot hope to find such counterexamples, as the following proposition suggests:

\begin{prop}
If $[F:F^p]=p$ then all the $p$-algebras of prime degree over $F$ are right linked.
\end{prop}

\begin{proof}
Let $F'=F[\lambda : \lambda^p=t]$ be the unique inseparable field extension of degree $p$ of $F$.
Every division $p$-algebra contains an inseparable field extension of degree $p$ of $F$, which must be isomorphic to $F'$, so every such algebra has a symbol presentation $[\alpha,t)_F$ for some $\alpha \in F$, and they are all right linked.
\end{proof}

Such fields include global and local fields (see \cite[Chapter 2, Lemma 2.7.2]{FriedJarden2008}) and fields of the form $K(t)$ or $K((t))$ where $K$ is perfect. It is important to note that there can be $p$-algebras over such fields of prime degree that do not share all the cyclic separable field extensions of the center. For example, take $K$ to be the perfect closure of the function field in two variables $E(\alpha,\beta)$ over some field $E$ of characteristic $p$, and take  $F=K(t)$ to be the function field in one variable over $K$. Then $[\alpha,t)_{p,F}$ contains $F[x : x^p-x=\alpha]$ but $[\beta,t)_{p,F}$ does not.
This means that sharing all the inseparable field extensions of the center does not imply sharing all the cyclic separable ones.

In order to construct examples of division $p$-algebras that are left linked but not right linked we must therefore turn to larger fields, and the place to start is function fields or Laurent series in two variables over perfect fields.

\begin{exmpl}\label{Mainexample}
The algebras $[1,\alpha)_{p,F}$ and $[1,\beta)_{p,F}$ are not right linked when $F=\mathbb{F}_p(\alpha,\beta)$ or $\mathbb{F}_p((\alpha))((\beta))$.
\end{exmpl}

\begin{proof}
It is enough to prove it for $F=\mathbb{F}_p((\alpha))((\beta))$ because in the other case we can extend scalars to this completion with respect to the $(\alpha,\beta)$-adic valuation. The value groups of these algebras are $\frac{1}{p}\mathbb{Z} \times \mathbb{Z}$ and $\mathbb{Z} \times \frac{1}{p}\mathbb{Z}$ respectively, and with residue division ring the field $\mathbb{F}_p[\lambda : \lambda^p-\lambda=1]$. If $L$ is a maximal subfield of $[1,\alpha)_{p,F}$, then since the rank of $L$ is $p$, one of the following holds: either $L/F$ is unramified, or it is totally ramified. If it is unramified, then its residue field is the residue field of the division algebra, hence $L$ is separable. So if L is purely inseparable, then it is totally ramified, with value group $\frac{1}{p}\mathbb{Z} \times \mathbb{Z}$. Similarly, any purely inseparable maximal subfield of $[1,\beta)_{p,F}$ is a totally ramified extension of the center, with value group $\mathbb{Z} \times \frac{1}{p}\mathbb{Z}$. Therefore, $[1,\alpha)_{p,F}$ and $[1,\beta)_{p,F}$ do not share any maximal subfield that is a purely inseparable extension of the center.
\end{proof}

In case $p=2$, the algebras $[1,\alpha)_{2,F}$ and $[1,\beta)_{2,F}$ over the function field $F=\mathbb{F}_2(\alpha,\beta)$ in two variables over $\mathbb{F}_2$ were exactly the algebras given in \citep{Lam:char2quat} as an example of a pair of left linked quaternion algebras which are not right linked. The proof here, however, is essentially different, making use of valuations instead of quadratic forms, in order to make it work for arbitrary prime numbers.

%

\section*{Acknowledgments}
I thank Jean-Pierre Tignol for his help and support, and especially for suggesting the proof of Example \ref{Mainexample}. I thank also Uzi Vishne and Adrian Wadsworth for their comments on the manuscript.

\section*{Bibliography}
\bibliographystyle{amsalpha}
\bibliography{bibfile}
\end{document}